\documentclass[11pt,openbib]{amsart}

\usepackage{amsmath}
\usepackage{amssymb}
\usepackage{amsthm}
\usepackage{amscd}
\usepackage{mathrsfs}

\newcommand{\ri}{\mathfrak{o}}
\newcommand{\mi}{\mathfrak{p}}

\newcommand{\Z}{\mathbf{Z}}
\newcommand{\C}{\mathbf{C}}

\newcommand{\supp}{\operatorname{supp}}

\newcommand{\e}{\sqrt{\epsilon}}
\newcommand{\p}{\varpi}

\newcommand{\U}{\mathrm{U}}

\theoremstyle{plain}
\newtheorem{thm}[equation]{Theorem}
\newtheorem{lem}[equation]{Lemma}
\newtheorem{prop}[equation]{Proposition}
\newtheorem{cor}[equation]{Corollary}
\theoremstyle{definition}
\newtheorem{defn}[equation]{Definition}

\newtheorem{rem}[equation]{Remark}

\def\Section#1{\section{#1}\setcounter{equation}{0}}

\title{On local newforms for unramified $\U(2,1)$}
\author{Michitaka Miyauchi}
\address{
Department of Mathematics, Faculty of Science\\
Kyoto University\\
Oiwake Kita-Shirakawa Sakyo Kyoto 606-8502 JAPAN
}
\email{miyauchi@math.kyoto-u.ac.jp}
\begin{document}
\maketitle
\pagestyle{myheadings}
\markboth{}{}

\begin{abstract} 
Let $G$ be the unramified unitary group in three variables
defined over a $p$-adic field $F$ with $p \neq 2$.
In this paper,
we investigate local newforms for irreducible admissible representations of $G$.
We introduce a family of open compact subgroups
$\{K_n\}_{n \geq 0}$
of $G$
to define the local newforms 
for representations of $G$
as the $K_n$-fixed vectors.
We prove the existence of local newforms for generic representations
and 
the multiplicity one property of the local newforms for admissible representations.
\end{abstract}

\section*{Introduction}
In the modern theory of automorphic representations,
local newforms play a very important role.
In fact,
for the study of special values of automorphic $L$-functions
through their integral presentations,
the existence of local newforms is crucial.
Besides the global application,
local newforms are indispensable for the ramification theory of
representations of $p$-adic groups.
It was Casselman who established the local newform theory for $\mathrm{GL}(2)$,
which can be stated as follows.
For a non-archimedean local field $F$ of characteristic zero
with ring of integers $\ri_F$
and its maximal ideal $\mi_F$,
the local counterpart of level subgroup of $\mathrm{GL}_2(F)$
is defined by
\[
\Gamma_0(\mi_F^n) 
= 
\left(
\begin{array}{cc}
\ri_F & \ri_F \\
\mi_F^n & 1+\mi_F^n
\end{array}
\right)^\times,
\]
for $n \geq 0$.
For each  irreducible admissible
representation $(\pi, V)$ of $\mathrm{GL}_2(F)$,
define the subspace
\[
V(n) = \{v \in V\, |\, \pi(k)v = v,\ k \in \Gamma_0(\mi_F^n)\},\
n \geq 0.
\]
Then 
the following was shown by Casselman \cite{Casselman}.
\begin{thm}[Local newforms for $\mathrm{GL}(2)$]\label{thm:gl2}
Let $(\pi, V)$ be an irreducible generic 
representation of $\mathrm{GL}_2(F)$.

(i) There is a non-negative integer $n$ such that
$V(n) \neq \{0\}$.

(ii) Put $c(\pi) = \min \{n\, |\, V(n) \neq \{0\}\}$.
Then $\dim V(c(\pi)) = 1$.

(iii)
For each $n \geq c(\pi)$,
we have
\[
\dim V(n) = n-c(\pi)+1.
\]

(iv)
For any non-zero element $v$ in $V(c(\pi))$,
the corresponding Whittaker function $W_v$ to $v$
is non-zero at $1$.
\end{thm}
The integer $c(\pi)$ is called {\it the conductor of} $\pi$.
It is also known that 
the $\varepsilon$-factor for $\pi$
 is a constant multiple of $q^{-sc(\pi)}$,
 where $q$ is the cardinality of the residue class field of $F$.
Similar results were obtained by Jacquet, Piatetski-Shapiro and Shalika \cite{JPSS}
and by Reeder \cite{Reeder} for $\mathrm{GL}_n(F)$.
Recently, Roberts and Schmidt \cite{RS}
established a theory of local newforms 
for irreducible representations of $\mathrm{GSp}_4(F)$
with trivial central character.
They considered paramodular subgroups of $\mathrm{GSp}_4(F)$
to define local newforms.
Our main concern is to construct a similar theory for 
unramified unitary group $\mathrm{U}(2,1)$.
We note that
for unitary group $\U(1,1)$,
there is a result by Lansky and Raghuram \cite{LR}, 
which 
determined the dimensions of the spaces of 
vectors fixed by certain open compact subgroups.
Unfortunately they do not concern the 
relation between their conductors and the exponents of 
$\varepsilon$-factors.

We assume that the residual characteristic of $F$ is odd.
Let $E$ be the unramified quadratic  extension over $F$.
Let $\ri_E$ denote the ring of integers in $E$,
$\mi_E$ the maximal ideal in $\ri_E$.
We realize our group $G$
as
$\{ g \in \mathrm{GL}_3(E)\ |\ 
{}^t\overline{g}Jg= J \}$
and denote it by $\mathrm{U}(2,1)(E/F)$,
where ${}^-$ is the non-trivial element in $\mathrm{Gal}(E/F)$
and
$J = 
\left(
\begin{array}{ccc}
0 & 0 &1\\
0 & 1 & 0\\
1 & 0 & 0
\end{array}
\right)$.
We define open compact subgroups $K_n$ of $G$
as an analog of paramodular subgroups of $\mathrm{GSp}_4(F)$;
\begin{eqnarray}
K_n
=
\left(
\begin{array}{ccc}
\ri_E & \ri_E &\mi^{-n}_E\\
\mi^n_E & 1+\mi^n_E & \ri_E\\
\mi^n_E & \mi^n_E & \ri_E
\end{array}
\right)
\cap G,
\end{eqnarray}
for $n \geq 0$.
We show the followings as the main results of this paper.
\begin{thm}[Multiplicity one of newforms]\label{thm:main}
Let $(\pi, V)$ be an irreducible generic representation of $G$.
We denote by $V(n)$ the space of $K_n$-fixed vectors in $V$.

(i) 
There is a non-negative integer $n$ such that
$V(n) \neq \{0\}$.

(ii) 
Put $N_\pi = \min \{n\, |\, V(n) \neq \{0\}\}$.
Then $\dim V(N_\pi) = 1$.
\end{thm}

We call $V(N_\pi)$ the space of {\it newforms},
and
$V(n)$ the space of {\it oldforms},
for $n > N_\pi$.
The part (ii) of Theorem~\ref{thm:main} can be shown 
for irreducible admissible representations with non-zero $K_n$-fixed vectors.
For the uniqueness of the newforms,
we do not assume any condition on the central characters.
In the case when
the central character is trivial 
in the neighborhood of newforms,
we can show more:
\begin{thm}[Dimensions of oldforms, test vectors for the Whittaker functional]
Let $(\pi, V)$ be an irreducible generic representation of $G$.
We denote by $n_\pi$ the conductor of the central character of $\pi$.
Suppose that $N_\pi >n_\pi$ and $N_\pi \geq 2$.
Then

(i) For any $n \geq N_\pi$, we have
\[
\dim V(n) = \left\lfloor \frac{n-N_\pi}{2}\right\rfloor +1.
\]

(ii)
A non-zero element $v$ in $V(N_\pi)$
is a test vector for the Whittaker functional,
that is,
$W_v(1) \neq 0$,
where
$W_v$ is the Whittaker function corresponding to $v$.
\end{thm}

We summarize the contents of this paper.
In section 1, 
we fix the basic notation for representations of the unramified unitary group
in three variables.
In section 2,
we introduce the notion of local newforms and 
prove that any irreducible generic representations of $G$
admit local newforms.
In section 3,
we define two level raising operators 
$\theta'$ and 
$\eta$
following Roberts and Schmidt \cite{RS}.
There
the $P_3$-theory plays an important role
to estimate the dimensions of the oldforms for $\mathrm{GSp}(4)$.
Here $P_n$ is the mirabolic subgroup of $\mathrm{GL}_n(F)$. 
In section 4,
we recall 
the $P_2$-theory for $\mathrm{U}(2,1)$
from Baruch \cite{Baruch},
and 
consider \lq\lq Kirillov model'' for generic representations of $G$.
In section 5,
we prove our main theorem,
that is,
multiplicity one theorem of local newforms (Theorem~\ref{thm:mult_main}).
Moreover,
we give the dimension formula of oldforms
for generic representations of $G$
whose
conductors are different from those of
their central characters (Theorem~\ref{thm:main3}).
We also show that all generic supercuspidal representations satisfy this condition.

We have not yet obtained the dimension formula of oldforms
for
representations whose 
conductors are equal to those of central characters.
Although we define the conductors
of generic representations of $G$,
here 
we do not consider  comparison of them with 
their $\varepsilon$-factors.
We hope to consider these problems in the future.

\medskip
\noindent
{\bf Acknowledgements} \
The author is grateful to Yoshi-hiro Ishikawa for
careful reading of the first draft of this paper,
and would like to thank 
Takuya Yamauchi 
and Tadashi Yamazaki for helpful discussions.
\Section{Preliminaries}
Here we realize our unramified unitary group in three variables
and 
summarize basic notation of its subgroups 
and terminology of its representations, which are used in this paper.
Let $F$ be a non-archimedean local field of characteristic zero.
Let $\ri_F$ be the ring of integers in $F$,
$\mi_F = \p_F \ri_F$ the maximal ideal in $\ri_F$ and
$k_F = \ri_F/\mi_F$ the residue field.
We denote by $q = q_F$ the cardinality of $k_F$.
Let $|\cdot|_F$ denote the absolute value of $F$
normalized so that $|\p_F|_F = q_F^{-1}$.
We use the analogous notation 
for any non-archimedean local field.
Throughout this paper, we  assume that
the characteristic $p$ of $k_F$ is odd.

Let $E = F[\e]$ be the unramified quadratic  extension over $F$,
where $\epsilon \in \ri_F^\times
\backslash (\ri_F^\times)^2$.
Then $\p_F$ is a uniformizer of $E$.
We abbreviate $\p = \p_F$.
We denote by ${}^-$ the non-trivial element in $\mathrm{Gal}(E/F)$.
We set
$G =
\{ g \in \mathrm{GL}_3(E)\ |\ 
{}^t \overline{g}Jg = J \}$
where
\begin{eqnarray*}
J = 
\left(
\begin{array}{ccc}
0 & 0 &1\\
0 & 1 & 0\\
1 & 0 & 0
\end{array}
\right).
\end{eqnarray*}
Then $G$ is the $F$-points of
the unramified unitary group $\mathrm{U}(2,1)$
over $F$.

Let $B$ denote the Borel subgroup of $G$ consisting of 
the upper triangular matrices in $G$.
We denote by $T$ the subgroup of $B$ consisting of the diagonal matrices in $G$.
The unipotent radical $U$ of $B$ is given by
\[
U = \left\{ u(x,y) =
\left(
\begin{array}{ccc}
1 & x & y\\
0 & 1 & -\overline{x}\\
0 & 0 & 1
\end{array}
\right)\, 
\Bigg|\,
x, y \in E,\
y+\overline{y} +x\overline{x} = 0
\right\}.
\]
We denote the opposite of $U$ by $\hat{U}$;
\[
\hat{U} = \left\{ \hat{u}(x,y) =
\left(
\begin{array}{ccc}
1 & 0 & 0\\
x & 1 & 0\\
y & -\overline{x} & 1
\end{array}
\right)\, 
\Bigg|\,
x, y \in E,\
y+\overline{y} +x\overline{x} = 0
\right\}.
\]

We embed the group $H = U(1,1)(E/F)$ into $G$ as
\begin{eqnarray*}
H = \left\{
\left(
\begin{array}{ccc}
a & 0 & b\\
0 & 1 & 0\\
c & 0 & d
\end{array}
\right) \in G
\right\}.
\end{eqnarray*}
We put $B_H = B\cap H$, $\hat{U}_H =\hat{U}\cap H$,
\begin{eqnarray*}
& U_H = U\cap H = \left\{
u(y) = \left(
\begin{array}{ccc}
1 & 0 & y\\
0 & 1 & 0\\
0 & 0 & 1
\end{array}
\right)\, \Bigg|\,
y \in E,\ y+\overline{y} = 0
\right\}
\end{eqnarray*}
and
\begin{eqnarray*}
& T_H = T\cap H = \left\{
t(a) = \left(
\begin{array}{ccc}
a & 0 & 0\\
0 & 1 & 0\\
0 & 0 & \overline{a}^{-1}
\end{array}
\right)\, \Bigg|\,
a \in E^\times
\right\}.
\end{eqnarray*}

We fix a non-trivial additive character $\psi_E$ of $E$
with conductor $\ri_E$,
and define a character $\psi$ of $U$ by
\[
\psi(u) = \psi_E(x),\ \mathrm{for}\
u = u(x,y) \in U.
\]
For  any irreducible admissible representation $(\pi, V)$
of $G$,
it is well-known that 
\[
\dim \mathrm{Hom}_U(\pi, \psi) \leq 1.
\]
We say that $(\pi, V)$ is {\it generic} if
$\mathrm{Hom}_U(\pi, \psi) \neq \{0\}$.
If $(\pi, V)$ is an irreducible generic representation of $G$,
then
by Frobenius reciprocity, we have
\[
\mathrm{Hom}_U(\pi, \psi) 
\simeq \mathrm{Hom}_G(\pi, \mathrm{Ind}_U^G \psi) \simeq \C.
\]
So there exists a unique embedding of
$\pi$ into $\mathrm{Ind}_U^G \psi$
up to scalar.
The image $\mathcal{W}(\pi, \psi)$ of $V$ is called {\it the Whittaker model of} $\pi$.
By a non-zero functional $l \in \mathrm{Hom}_U(\pi, \psi)$,
which 
is called {\it the Whittaker functional},
we define {\it the Whittaker function} $W_v \in \mathcal{W}(\pi, \psi)$
associated to $v \in V$ by
\[
W_v(g) = l(\pi(g)v),\ g \in G.
\]

There is an isomorphism $\iota$
between  the center $Z$ of $G$
and the norm-one subgroup $E^{1}$ of $E^\times$,
given by
\[
\iota:
E^{1} \simeq Z;
\lambda \mapsto 
\left(
\begin{array}{ccc}
\lambda & & \\
& \lambda & \\
& & \lambda
\end{array}
\right).
\]
We set open compact subgroups
of $E^{1}$ as
\[
E^{1}_0 = E^{1},\
E^{1}_n = E^{1}\cap (1+\mi_E^n),\ \mathrm{for}\ n \geq 1.
\]
Then $\{E_n^{1}\}_{n \geq 0}$ gives a filtration of 
$E^{1}$.
For an irreducible admissible representation $(\pi, V)$ of $G$,
we denote by $\omega_\pi$ the central character of $\pi$.
We define {\it the conductor} $n_\pi$ {\it of} $\omega_\pi$
by
\begin{eqnarray*}
n_\pi = \mathrm{min}\{n \geq 0\, |\, \omega_\pi|_{Z_n} = 1\},
\end{eqnarray*}
where  $Z_n = \iota(E^{1}_n)$.
\Section{Local newforms}
In this section,
we introduce the notion of newforms for representations of $G$.
A newform for an irreducible admissible representation of $G$
is a vector which is fixed by a certain open compact subgroup of $G$.
We prove that every 
irreducible generic representation of $G$
admits a newform (Theorem~\ref{thm:gen_new}).

We introduce a family of  open compact subgroups $\{K_n\}_{n \in \Z_{\geq 0}}$ of $G$
by
\begin{eqnarray*}
K_n
=
\left(
\begin{array}{ccc}
\ri_E & \ri_E &\mi^{-n}_E\\
\mi^n_E & 1+\mi^n_E & \ri_E\\
\mi^n_E & \mi^n_E & \ri_E
\end{array}
\right)
\cap G,
\end{eqnarray*}
which is used to define our local newforms.
We set
\begin{eqnarray*}
t_{n}
= \left(
\begin{array}{ccc}
 & & \p^{-n}\\
 & 1 & \\
 \p^{n} & &
\end{array}
\right) \in K_n.
\end{eqnarray*}
The group $K_0$ is a good maximal compact subgroup of $G$.
So we have the Iwasawa decomposition $G =BK_0$.
Moreover, we obtain the the following decomposition
for $K_1$:
\begin{lem}\label{lem:K_1}
$G = B K_1$.
\end{lem}
\begin{proof}
The quotient $K_0/(1+M_3(\mi_E))\cap G$ is isomorphic to
$\mathrm{U}(2,1)(k_E/k_F)$.
Using the Bruhat decomposition of $\U(2,1)(k_E/k_F)$,
we get
\begin{eqnarray}\label{eq:decomp}
G = BK_0 = BW(K_0\cap K_1),
\end{eqnarray}
where $W = \{1, t_0\}$.
Since $t_0 \in Bt_1 \subset BK_1$,
we obtain $G = BK_1$.
\end{proof}

We define an open compact subgroup $U(\ri_E)$ of $U$
by
\begin{eqnarray*}
U(\ri_E)
=
\left(
\begin{array}{ccc}
1 & \ri_E & \ri_E\\
0 & 1 & \ri_E\\
0 & 0 & 1
\end{array}
\right)\cap G.
\end{eqnarray*}
\begin{lem}\label{lem:structure1}
For $n \geq 0$,
the group $K_n$ is generated by $K_n\cap H$ and $U(\ri_E)$.
\end{lem}
\begin{proof}
Let $K'$ denote the subgroup of $G$
generated by $K_n\cap H$ and $U(\ri_E)$.
We prove that $K'$ contains $K_n$.

(i) Suppose that $n \geq 1$.
Put
\[
K'' = 
\left(
\begin{array}{ccc}
1+\mi^n_E & \ri_E &\ri_E\\
\mi^n_E & 1+\mi^n_E & \ri_E\\
\mi^{2n}_E & \mi^n_E & 1+\mi^n_E
\end{array}
\right)
\cap G.
\]
It is easy to check that
$K_n = (K_n\cap H) K''$.
So it is enough to prove that $K'' \subset K'$.

The group $K''$ has an Iwahori decomposition
$K'' = (K''\cap \hat{U})(K''\cap T)(K''\cap U)$.
The group $K'$ contains $K''\cap U = U(\ri_E)$.
Since $t_n \in K_n \cap H$, 
$K'$ contains $K''\cap \hat{U} = t_n U(\ri_E) t_n$.
We have $K''\cap T = (K''\cap T_H) (K''\cap Z)
= (K''\cap T_H)Z_n$.
Since $K''\cap T_H \subset K_n \cap H \subset K'$,
it suffices to show $Z_n\subset K'$.
We note that
\begin{eqnarray}\label{eq:deco}
K_n\cap U,\, K_n\cap \hat{U} \subset K'
\end{eqnarray}
because
$K_n\cap U = (K''\cap U)(K_n\cap U_H)$
and 
$K_n\cap \hat{U} = (K''\cap \hat{U})(K_n\cap \hat{U}_H)$.

We shall prove $Z_n \subset K'$.
Since $Z_n \neq Z_{n+1}$,
it is enough to check $Z_n\backslash Z_{n+1} \subset K'$.
Let $1+x$ be an element in $E_n^{1}\backslash E_{n+1}^{1}$.
Then there is $z \in \mi_E^{-n}\backslash \mi_E^{-1-n}$
such that
$1+x = -\overline{z}/z$.
Since $z+\overline{z} = -xz\in \ri_F^\times$,
there is an element $y \in \ri_E^\times$ 
such that $z+\overline{z}+y \overline{y} = 0$.
We have
\begin{eqnarray}\label{eq:d}
\mathrm{diag}(\p^n z, -\overline{z}/z, \p^{-n}\overline{z}^{-1} )
= \hat{u}(\overline{y}/z, 1/\overline{z})u(y, z)
\hat{u}(\overline{y}/\overline{z}, 1/\overline{z}) t_n.
\end{eqnarray}
By (\ref{eq:deco}),
all elements in the right-hand side in (\ref{eq:d}) belong to $K'$.
So
we get 
\[
\mathrm{diag}(\p^n z, -\overline{z}/z, \p^{-n}\overline{z}^{-1} ) \in K'.
\]
Since $\p^n {z} \in \ri_E^\times$,
we obtain
$\iota(1+x) = \iota(-\overline{z}/z) = 
t(-\p^{-n} \overline{z}/z^2)
\mathrm{diag}(\p^n z, -\overline{z}/z, \p^{-n}\overline{z}^{-1} ) 
\in K'$.
This completes the proof for $n \geq 1$.

(ii) Suppose that $n = 0$.
By (\ref{eq:decomp}),
we get $K_0 = (B\cap K_0) W (K_0\cap K_1)$.
We have $W = \{1, t_0\} \subset K_0\cap H \subset K'$
and
$K_0 \cap U \subset K'$.
So we get
$K_0 \cap \hat{U} = t_0 (K_0\cap U) t_0 \subset K'$.
Since $B\cap K_0 = (K_0\cap T)(K_0\cap U)$
and
$K_0\cap K_1$ has an Iwahori decomposition,
it suffices to prove $K_0\cap T \subset K'$.
Note that $K_0\cap T = (K_0\cap T_H)Z$
and $K_0 \cap T_H\subset K'$.
So it is enough to prove $Z \subset K'$.

Since $Z =Z_0 \neq Z_1$,
it suffices to prove $Z\backslash Z_{1} \subset K'$.
Let $x$ be an element in $E^{1}\backslash E_{1}^{1}$.
Then there is $z \in \ri_E^\times$
such that
$x = -\overline{z}/z$.
Since $z+\overline{z} = (1-x)z \in \ri_F^\times$,
there is an element $y \in \ri_E^\times$ such that
$z+\overline{z}+y \overline{y} = 0$.
So we can observe $\iota(x) = \iota(-\overline{z}/z) \in K'$
as in the case when $n \geq 1$.
\end{proof}

Let $(\pi, V)$ be an irreducible admissible representation of $G$.
For each non-negative integer $n$, we define a subspace
\begin{eqnarray*}
V(n) = \{v \in V\ |\ \pi(k)v= v,\ k \in K_n \}
\end{eqnarray*}
of $V$.
Since $\pi$ is admissible, $V(n)$ is finite-dimensional for all $n \geq 0$.
\begin{defn}
Let $(\pi, V)$ be an irreducible admissible representation of $G$
which has $K_n$-fixed vectors for some $n \geq 0$.
We define {\it the conductor of} $\pi$ by
\begin{eqnarray*}
N_\pi = \mathrm{min}\{n \geq 0\, |\, V(n) \neq \{0\}\}.
\end{eqnarray*}
We call the vectors in $V(N_\pi)$ {\it the newforms} for $\pi$,
and all 
the elements in $V(n)$ {\it the oldforms} for $\pi$,
for $n > N_\pi$.
\end{defn}

\begin{rem}
Since $Z_n = Z\cap K_n$,
the central character $\omega_\pi$ of $\pi$ is trivial on $Z_n$
if $V(n) \neq \{0\}$.
This implies
\begin{eqnarray*}
N_\pi \geq n_\pi.
\end{eqnarray*}
The relation between these conductors plays an important role 
in section~\ref{sec:main}.
\end{rem}

The following theorem states that
we can define the conductors at least  for generic representations
of $G$.
\begin{thm}\label{thm:gen_new}
If an irreducible admissible representation $(\pi, V)$ of $G$
is generic,
then there exists a non-negative integer $n$ such that $V(n) \neq \{0\}$.
\end{thm}
\begin{proof}
We regard $|\cdot|_E$ as a quasi-character of $T_H \simeq E^\times$.
It follows from
\cite{Baruch} Theorem 4.7
that
$\dim \mathrm{Hom}_H(V, \mathrm{Ind}_{B_H}^H |\cdot|_E^s) = 1$
, outside a finite number of values  of $q^{2s}$.
So we can choose $s \in \C$ such that 
$\dim \mathrm{Hom}_H(V, \mathrm{Ind}_{B_H}^H |\cdot|_E^s) = 1$
and $\mathrm{Ind}_{B_H}^H |\cdot|_E^s$ is an unramified principal series representation of $H \simeq U(1,1)(E/F)$.
Thus there exists a non-zero $K_0\cap H$-fixed vector $v$
in $V$.
Take a positive integer $n$ so that
$v$ is fixed by
\[
\left(
\begin{array}{ccc}
1 & \mi_E^n &  \mi_E^n\\
 & 1 &\mi_E^n\\
 & & 1
\end{array}
\right)\cap G.
\]
Then it follows from Lemma~\ref{lem:structure1} that 
the vector
\[
\pi 
\left(
\begin{array}{ccc}
\p^{-n} & & \\
& 1 & \\
& & \p^n
\end{array}
\right)v
\]
lies in $V(2n)$. 
\end{proof}

\Section{Level raising operators}
Let $(\pi, V)$ be an irreducible admissible representation of $G$.
In this section, we define 
two level raising operators 
$\eta: V(n) \rightarrow V(n+2)$ and
$\theta': V(n) \rightarrow V(n+1)$ as in \cite{RS} subsection 3.2,
and study their several properties.

We set 
\begin{eqnarray*}
\eta = 
\left(
\begin{array}{ccc}
\p^{-1} & & \\
& 1 & \\
& & \p
\end{array}
\right).
\end{eqnarray*}
For $n \geq 0$,
we have $K_{n+2} \subset \eta K_n\eta^{-1}$.
So we can define an operator $\eta: V(n) \rightarrow V(n+2)$
by 
\begin{eqnarray*}
\eta v = \pi(\eta) v,\ v \in V(n).
\end{eqnarray*}
We define an open compact subgroup 
$U(\mi_E^{-1})$ of $U$ by
\begin{eqnarray*}
U(\mi_E^{-1})
=
\left(
\begin{array}{ccc}
1 & \mi_E^{-1}& \mi_E^{-2}\\
0 & 1 & \mi_E^{-1}\\
0 & 0 & 1
\end{array}
\right)\cap G.
\end{eqnarray*}
\begin{lem}\label{lem:structure2}
Let $n$ be a non-negative integer
and
$v$ an element in $V(n+2)$.
Then $v \in \eta V(n)$ if and only if 
$v$ is fixed by $U(\mi_E^{-1})$.
\end{lem}
\begin{proof}
Observe that
$\eta (K_n\cap H)\eta^{-1}= K_{n+2}\cap H$ 
and $
\eta U(\ri_E)\eta^{-1} = U(\mi_E^{-1})$.
Lemma~\ref{lem:structure1} says
that the group $\eta K_n\eta^{-1}$
is generated by $K_{n+2}\cap H$ and $
U(\mi_E^{-1})$.
The assertion follows immediately from this.
\end{proof}

We define another level raising operator $\theta': V(n) \rightarrow V(n+1)$
by
\begin{eqnarray*}
\theta' v 
= \frac{1}{\mathrm{vol}(K_{n+1}\cap K_n)}
\int_{K_{n+1}} \pi(k) v dk,\ v \in V(n).
\end{eqnarray*}
To describe $\theta'$ explicitly,
we prepare the following:
\begin{lem}\label{lem:coset}
Let $n$ be a non-negative integer.
Then a complete system of representatives for $K_{n+1}/K_{n+1}\cap K_n$
is given by $q+1$ elements
$t_{n+1}$
{and}
$u(x\e)$,
$x \in \mi_F^{-1-n}/\mi_F^{-n}$.
\end{lem}
\begin{proof}
We have
\[
K_{n+1}\cap K_n
=
\left(
\begin{array}{ccc}
\ri_E & \ri_E &\mi^{-n}_E\\
\mi^{n+1}_E & 1+\mi^{n+1}_E & \ri_E\\
\mi^{n+1}_E & \mi^{n+1}_E & \ri_E
\end{array}
\right)\cap G.
\]
It is easy to observe that 
the elements in the assertion belong to pairwise distinct cosets
in $K_{n+1}/K_{n+1}\cap K_n$.
We write the $(i,j)$-entry of $k \in K_{n+1}$
as $k_{ij}$.
Suppose that $k_{33} \in \mi_E$.
Then we have $t_{n+1}k \in K_{n+1}\cap K_n$,
and hence $k \in t_{n+1}(K_{n+1}\cap K_n)$.
If $k_{33} \in \ri_E^\times$,
then
we have
$k_{13}\overline{k}_{33}+
k_{23}\overline{k}_{23}
+\overline{k}_{13}{k}_{33} = 0$ because $k$ lies in $G$.
This implies $k_{13}\overline{k}_{33} \in \ri_F \oplus \mi_F^{-1-n}\e$.
Since $k_{33}\overline{k}_{33} \in \ri_F^\times$,
we get
$k_{13}{k}^{-1}_{33} \in \ri_F \oplus \mi_F^{-1-n}\e$.
So there exists $x \in \mi_F^{-1-n}$ such that
$k_{13}-x\e k_{33} \in \ri_E$.
We therefore have
$
k \in
u(x\e)
(K_{n+1}\cap K_n)$.
\end{proof}

\begin{prop}\label{prop:theta}
Let $n$ be a non-negative integer.
Then we have
\[
\theta' v
= \eta v +
\sum_{x \in \mi_F^{-1-n}/\mi_F^{-n}}
\pi\left(
\begin{array}{ccc}
1 & & x\e\\
 & 1 & \\
  &  & 1
\end{array}
\right)v,\ v \in V(n).
\]
\end{prop}
\begin{proof}
The proposition follows from 
Lemma~\ref{lem:coset} and
the equation $t_{n+1} = \eta t_n$.
\end{proof}

By Proposition~\ref{prop:theta},
the operators $\theta'$ and $\eta$ commute
each other.
\begin{cor}
Let $n$ be a non-negative integer.
We have $\eta \theta' v = \theta' \eta v$,
for all $v \in V(n)$.
\end{cor}

We prepare three Lemmas~\ref{lem:325},
\ref{lem:eta} and \ref{lem:u},
whose proofs are similar to those of
Theorems 3.2.5, 3.2.6 and Lemma 3.4.1 in \cite{RS}
respectively.

\begin{lem}\label{lem:325}
Let $n$ be a positive integer.
Let $v$ be an element in $V$ such that
\[
\sum_{x \in \mi_F^{-1-n}/\mi_F^{-n}}
\pi\left(
\begin{array}{ccc}
1 & & x\e\\
& 1 & \\
& & 1
\end{array}
\right)v = 0.
\]
Suppose that
$v$ is fixed by the following subgroups of $G$:
\begin{eqnarray*}
(i)
\left(
\begin{array}{ccc}
1 & & \mi_E^{-n}\\
& 1 & \\
& & 1
\end{array}
\right)\cap G,\
(ii)
\left(
\begin{array}{ccc}
\ri_E^\times & & \\
& 1 & \\
& & \ri_E^\times
\end{array}
\right)\cap G,\
(iii)
\left(
\begin{array}{ccc}
1 & & \\
\mi_E^n & 1 & \\
\mi_E^{n+1} & \mi_E^n & 1
\end{array}
\right)\cap G.
\end{eqnarray*}
Then $v$ is fixed by $t_{n+1}$ and $U(\mi_E^{-1})$.
\end{lem}
\begin{proof}
Since $v$ is fixed by the subgroup (i),
the sum
\[
\sum_{x \in \mi_F^{-1-n}/\mi_F^{-n}}
\pi\left(
\begin{array}{ccc}
1 & & x\e\\
& 1 & \\
& & 1
\end{array}
\right)v
\]
is well-defined.

We claim that $v$ is fixed by $t_{n+1}$.
By assumption, we have
\[
-v = 
\sum_{
\substack{x \in \mi_F^{-1-n}/\mi_F^{-n}\\x\not\equiv 0}}
\pi\left(
\begin{array}{ccc}
1 & & x\e\\
& 1 & \\
& & 1
\end{array}
\right)v.
\]
Because $v$ is fixed by the subgroup (ii) and (iii),
we obtain
\begin{eqnarray*}
& & -\pi(t_{n+1})v  = 
\sum_{
\substack{x \in \mi_F^{-1-n}/\mi_F^{-n}\\x\not\equiv 0}}
\pi\left(t_{n+1}\left(
\begin{array}{ccc}
1 & & x\e\\
& 1 & \\
& & 1
\end{array}
\right)\right)v\\
& = & 
\sum_{
\substack{x \in \mi_F^{-1-n}/\mi_F^{-n}\\x\not\equiv 0}}
\pi\left(\left(
\begin{array}{ccc}
1 & & \p^{-2-2n}x^{-1}\e^{-1}\\
& 1 & \\
& & 1
\end{array}
\right)
\left(
\begin{array}{ccc}
-\p^{-1-n}x^{-1}\e^{-1} & & \\
& 1 & \\
\p^{n+1}& & \p^{n+1}x\e
\end{array}
\right)\right)v\\
& = & 
\sum_{
\substack{x \in \mi_F^{-1-n}/\mi_F^{-n}\\x\not\equiv 0}}
\pi\left(
\begin{array}{ccc}
1 & & \p^{-2-2n}x^{-1}\e^{-1}\\
& 1 & \\
& & 1
\end{array}
\right)v\\
& = & -v.
\end{eqnarray*}
Therefore, $v$ is fixed by $t_{n+1}$.

Since
\[
U(\mi_E^{-1})
= 
t_{n+1}
\left(\left(
\begin{array}{ccc}
1 & & \\
\mi_E^n & 1 & \\
\mi_E^{2n} & \mi_E^n & 1
\end{array}
\right)\cap G\right)
t_{n+1}
\]
and $v$ is fixed by the subgroup (iii),
we see that $v$ is fixed by $U(\mi_E^{-1})$.
\end{proof}

We introduce one more operator $S$ on $V$;
\begin{eqnarray*}
Sv 
= 
\frac{1}{\mathrm{vol}(U(\ri_E))}
\int_{U(\mi_E^{-1})}\pi(u)v du,\ v \in V.
\end{eqnarray*}
Let $n$ be a  non-negative integer and
let $v \in V(n)$.
One can observe that
$v$ is fixed by $U(\mi_E^{-1})$
if and only if $Sv = q^4 v$.
If $n \geq 2$,
it follows from Lemma~\ref{lem:structure2}
that $v \in \eta V(n-2)$
if and only if $Sv = q^4 v$.

\begin{lem}\label{lem:eta}
Let $n$ be an integer such that $n \geq 2$
and $n > n_\pi$.
Suppose that an element $v$ in $V(n)$ satisfies
$\theta'v \in \eta V(n-1)$.
Then $v \in \eta V(n-2)$.
\end{lem}
\begin{proof}
By assumption
and Proposition~\ref{prop:theta},
we have
\begin{eqnarray*}
0 & = & (S-q^4)
\theta' v
= (S-q^4)(\eta v +
\sum_{x \in \mi_F^{-1-n}/\mi_F^{-n}}
\pi\left(
\begin{array}{ccc}
1 & & x\e\\
 & 1 & \\
  &  & 1
\end{array}
\right)v)\\
& =& 
(S-q^4)
\sum_{x \in \mi_F^{-1-n}/\mi_F^{-n}}
\pi\left(
\begin{array}{ccc}
1 & & x\e\\
 & 1 & \\
  &  & 1
\end{array}
\right)v\\
& =& 
\sum_{x \in \mi_F^{-1-n}/\mi_F^{-n}}
\pi\left(
\begin{array}{ccc}
1 & & x\e\\
 & 1 & \\
  &  & 1
\end{array}
\right)(S-q^4)v
\end{eqnarray*}
because $\eta v$ is fixed by $U(\mi_E^{-1})$
and $U_H$ is the center of $U$.

We claim that
$(S-q^4)v$ is fixed by the subgroups (i)-(iii) in Lemma~\ref{lem:325}.
Then  Lemma~\ref{lem:325}
says that $(S-q^4)v$ is fixed by $U(\mi_E^{-1})$.
Since $Sv$ is fixed by $U(\mi_E^{-1})$,
we see that $v$ is fixed by $U(\mi_E^{-1})$.
Therefore
we get $v \in \eta V(n-2)$
by Lemma~\ref{lem:structure2}.

We shall prove the claim.
Since $v$ is fixed by the subgroups
(i)-(iii) in Lemma~\ref{lem:325},
it is enough to check that 
$Sv$ is fixed by them.
It is obvious that 
$Sv$ is fixed by the subgroups (i) and (ii).
We shall show that
$Sv$ is fixed by the subgroup (iii).
Since $U(\ri_E)$ normalizes
$(1+M_3(\mi_E^{n-1}))\cap G$,
the group
$U(\mi_E^{-1}) = \eta U(\ri_E) \eta^{-1}$
normalizes
\[
K' = 
\eta (1+M_3(\mi_E^{n-1}))\eta^{-1}\cap G
=
\left(
\begin{array}{ccc}
1+\mi_E^{n-1} & \mi_E^{n-2} & \mi_E^{n-3}\\
\mi_E^n & 1+\mi_E^{n-1} & \mi_E^{n-2} \\
\mi_E^{n+1} & \mi_E^n & 1+\mi_E^{n-1} 
\end{array}
\right)\cap G.
\]
Since the subgroup (iii) lies in $K'$,
it is enough to prove that 
$v$ is fixed by $K'$.
Note that $K' \subset Z_{n-1}K_n$.
Since we are assuming $n-1 \geq n_\pi$,
the group
$Z_{n-1}$ acts on $V$ trivially.
We  therefore conclude that 
$v$ is fixed by $K'$.
This completes the proof.
\end{proof}

\begin{lem}\label{lem:u}
Let $n$ and $k$ be non-negative integers.
For $v \in V(n)$,
there exist $v_1 \in V(n+k-2)$ and $v_2 \in V(n+k-1)$
such that
\[
\int_{\mi_F^{-k-n}/\mi_F^{-n}}
\pi\left(
\begin{array}{ccc}
1 & & x\e\\
& 1 & \\
& & 1
\end{array}
\right)vdx = \theta^{'k} v + \eta v_1 + \eta v_2.
\]
Here, we put $v_1 = 0$ if $n+k-2< 0$
and $v_2 = 0$ if $n+k-1< 0$.
\end{lem}

\begin{proof}
We shall prove the lemma by induction on $k$.
Suppose that $k = 0$.
Then the assertion is true with $v_1 = v_2 = 0$.

Suppose that $k>0$.
Then, by the induction hypothesis, 
we have
\begin{eqnarray*}
& & 
\int_{\mi_F^{-k-n}/\mi_F^{-n}}
\pi\left(
\begin{array}{ccc}
1 & & x\e\\
& 1 & \\
& & 1
\end{array}
\right)vdx  \\
 & = & 
\int_{\mi_F^{-k-n}/\mi_F^{1-k-n}}
\pi\left(
\begin{array}{ccc}
1 & & y\e\\
& 1 & \\
& & 1
\end{array}
\right)
(
\int_{\mi_F^{1-k-n}/\mi_F^{-n}}
\pi\left(
\begin{array}{ccc}
1 & & x\e\\
& 1 & \\
& & 1
\end{array}
\right)vdx) dy\\
& = & 
\int_{\mi_F^{-k-n}/\mi_F^{1-k-n}}
\pi\left(
\begin{array}{ccc}
1 & & y\e\\
& 1 & \\
& & 1
\end{array}
\right)
( \theta^{'k-1} v + \eta v'_1 + \eta v'_2) dy,
\end{eqnarray*}
for some
$v'_1 \in V(n+k-3)$ and $v'_2 \in V(n+k-2)$.
By Proposition~\ref{prop:theta}
and the fact $\eta v'_2 \in  V(n+k)$,
we get
\begin{eqnarray*}
\int_{\mi_F^{-k-n}/\mi_F^{-n}}
\pi\left(
\begin{array}{ccc}
1 & & x\e\\
& 1 & \\
& & 1
\end{array}
\right)vdx  
& = & 
(\theta'-\eta)( \theta^{'k-1} v + \eta v'_1)+q\eta v'_2\\
& = & 
\theta^{'k}v +\eta (\theta' v'_1+qv'_2)
+\eta (-\theta^{'k-1}v-\eta v'_1).
\end{eqnarray*}
Put $v_1 = \theta' v'_1+qv'_2$
and $v_2 = -\theta^{'k-1}v-\eta v'_1$.
This completes the proof.
\end{proof}
\Section{$P_2$-theory}
In subsection~\ref{subsec:p2},
we recall $P_2$-theory for $\mathrm{U}(2,1)$
from \cite{Baruch} section 4,
and relate this to the level raising operators
(see Lemma~\ref{lem:pS}).
In subsection~\ref{subsec:kirillov},
we develop 
$P_2$-theory to the \lq\lq Kirillov models''
for generic representations,
and prove
multiplicity one theorem of newforms for 
generic representations whose conductors differ from 
those of their central characters
(Theorem~\ref{thm:gen_1}).

\subsection{$P_2$-modules}\label{subsec:p2}
Let $P_2$ be the subgroup of $\mathrm{GL}_2(E)$ consisting of 
the matrices of the form
\[
\left(
\begin{array}{cc}
* & *\\
0 & 1
\end{array}
\right).
\]
We get an isomorphism $T_H U/U_H \simeq P_2$
from the map
\[
t(a) u(x, y) \mapsto 
\left(
\begin{array}{cc}
a & 0\\
0 & 1
\end{array}
\right)
\left(
\begin{array}{cc}
1 & x\\
0 & 1
\end{array}
\right),\ a \in E^\times,\ x, y \in E.
\]
For an irreducible admissible representation $(\pi, V)$ of 
$G$,
we set 
$V_{U_H} = V/\langle \pi(u)v-v\ |\ v \in V, u \in U_H\rangle$.
Then we can regard $V_{U_H}$ as a $P_2$-module.
We denote by $p$ the natural projection from $V$
to $V_{U_H}$.

We recall from \cite{BZ} section 5
the structure of $P_2$-modules.
We put 
\[
Z_2 = 
\left\{
\left(
\begin{array}{cc}
1 & x\\
 & 1
\end{array}
\right)\, \Bigg|\, x \in E
\right\}
\]
and
\[
V_{U_H}(Z_2)
=\langle \pi(z)v -v\ |\ v \in V_{U_H},\ z \in Z_2\rangle.
\]
Then $V_{U_H}(Z_2)$ is a $P_2$-submodule of $V_{U_H}$
and $V_{U_H}/V_{U_H}(Z_2) \simeq V_U$,
where $V_U$ is the unnormalized Jacquet module of $(\pi, V)$.
Let $V_{U, \psi}$ denote the twisted Jacquet module of 
$(\pi, V)$ with respect to $(U, \psi)$.
Then 
$V_{U_H}(Z_2)$ is isomorphic to $(\dim V_{U, \psi}) \cdot
\mathrm{ind}_{Z_2}^{P_2}(\psi)$,
where 
$\psi$ is the character of $Z_2$
defined by 
\[
\psi\left(
\begin{array}{cc}
1 & x\\
 & 1
\end{array}
\right) = \psi_E(x),\ x \in E
\]
and
$\mathrm{ind}$ denotes the compactly supported induction.
We note that $\mathrm{ind}_{Z_2}^{P_2}(\psi)$
is an irreducible $P_2$-module.

We summarize a criterion of genericity and supercuspidality
of $\pi$ in terms of $P_2$-modules:
\begin{prop}\label{prop:P_2}
Let $(\pi, V)$ be an irreducible admissible representation of
$G$.

(i)
$\pi$ is supercuspidal if and only if $V_{U_H}
= V_{U_H}(Z_2)$;

(ii) $\pi$ is generic if and only if $V_{U_H}(Z_2) \neq \{0\}$.
Moreover, if this is the case, then
$V_{U_H}(Z_2) \simeq \mathrm{ind}_{Z_2}^{P_2}(\psi)$.
\end{prop}

The following lemma gives a criterion 
whether or not, a $K_n$-fixed vector 
comes via the operator $\eta$,
which is our main tool to prove the uniqueness of newform.
\begin{lem}\label{lem:pS}
Let $(\pi, V)$ be an irreducible admissible representation of
$G$
and let $n$ be an integer such that $n \geq 2$ and $n > n_\pi$.
Suppose that an element $v$ in $V(n)$
satisfies $p((S-q^4)v) = 0$.
Then $v$ belongs to $\eta V(n-2)$.
\end{lem}
\begin{proof}
The assumption
implies
$(S-q^4)v \in \langle \pi(u)v-v\ |\ v \in V, u \in U_H\rangle$.
So there exists a non-negative integer $k$ such that
\[
\int_{\mi_F^{-k-n}}
\pi\left(
\begin{array}{ccc}
1 & & x\e\\
& 1 & \\
& & 1
\end{array}
\right)(S-q^4)vdx = 0.
\]
Since $U_H$ is the center of $U$,
we get 
\[
(S-q^4)\int_{\mi_F^{-k-n}/\mi_F^{-n}}
\pi\left(
\begin{array}{ccc}
1 & & x\e\\
& 1 & \\
& & 1
\end{array}
\right)vdx = 0.
\]
It follows from Lemma~\ref{lem:u},
there are $v_1 \in V(n+k-2)$ and $v_2 \in V(n+k-1)$ such that
\[
\int_{\mi_F^{-k-n}/\mi_F^{-n}}
\pi\left(
\begin{array}{ccc}
1 & & x\e\\
& 1 & \\
& & 1
\end{array}
\right)vdx = \theta^{'k} v + \eta v_1 + \eta v_2.
\]
So we have
\[
(S-q^4)(\theta^{'k} v + \eta v_1 + \eta v_2)
= (S-q^4)\theta^{'k}v = 0.
\]
This implies 
$\theta^{'k}v \in \eta V(n+k-2)$.

If $k = 0$, then we have $v \in \eta V(n-2)$, as required.
Suppose that $k > 0$.
Then, by Lemma~\ref{lem:eta}, we get $\theta^{'k-1}v \in \eta V(n+k-3)$.
By repeating this argument,
we conclude that $v \in \eta V(n-2)$.
\end{proof}

In the remaining of this subsection,
we apply Lemma~\ref{lem:pS}
to non-generic representations.

\begin{lem}\label{lem:non-generic}
Let $(\pi, V)$ be an irreducible non-generic representation of 
$G$.
Then we have
$V(n) = \eta V(n-2)$, for any integer $n$
such that $n \geq 2$
and $n > n_\pi$.
\end{lem}
\begin{proof}
It suffices to prove that $V(n) \subset \eta V(n-2)$.
Due to Proposition~\ref{prop:P_2} (ii),
we have
$V_{U_H}(Z_2) = \{0\}$.
For $v \in V(n)$,
we see that $p((S-q^4)v) \in V_{U_H}(Z_2) = \{0\}$
since $U/U_H \simeq Z_2$.
Thus we get $v \in \eta V(n-2)$
by Lemma~\ref{lem:pS}.
\end{proof}

\begin{thm}\label{thm:non_gen}
If an irreducible non-generic representation $(\pi, V)$
of $G$
admits a newform,
then we have
\[
N_\pi = 0, 1,\ or\ N_\pi = n_\pi.
\]
\end{thm}
\begin{proof}
Suppose that $N_\pi \geq 2$ and $N_\pi > n_\pi$.
Then by Lemma~\ref{lem:non-generic},
we have
$V(N_\pi) = \eta V(N_\pi-2) =\{0\}$.
This contradicts the choice of $N_\pi$.
So we have $N_\pi < 2$ or $N_\pi = n_\pi$.
\end{proof}

\subsection{\lq\lq Kirillov model''}\label{subsec:kirillov}
In the remaining of this section,
we assume that
$(\pi, V)$ is an irreducible generic representation 
of $G$.
For $v \in V$, we define a function $\varphi_v$ on $E^\times$
by
\begin{eqnarray*}
\varphi_v(a)
=
W_v\left(
\begin{array}{ccc}
a & &\\
& 1 & \\
& & \overline{a}^{-1}
\end{array}
\right),\ a \in E^\times.
\end{eqnarray*}
Let $\mathcal{C}^\infty(E^\times)$ denote the space
of locally constant functions on $E^\times$.
Then we get a map $V \rightarrow \mathcal{C}^\infty(E^\times);
v \mapsto \varphi_v$.
It is easy to observe
that $\langle \pi(u)v-v\ |\ v \in V, u \in U_H\rangle$
lies in the kernel of this map.
We therefore obtain a map 
$V_{U_H} \rightarrow \mathcal{C}^\infty(E^\times);
p(v) \mapsto \varphi_v$.
We define an action of $P_2$ on 
$\mathcal{C}^\infty(E^\times)$ by
\begin{eqnarray*}
\left(
\left(
\begin{array}{cc}
a & x\\
0 & 1
\end{array}
\right)
\varphi
\right)(b)
= \psi_E(bx)\varphi(ab),\ a, b \in E^\times, x \in E.
\end{eqnarray*}
Then the map
$V_{U_H} \rightarrow \mathcal{C}^\infty(E^\times);
p(v) \mapsto \varphi_v$ is a $P_2$-homomorphism.

\begin{lem}\label{lem:varphi}
For $v \in V$,
there exists a non-negative integer $m$ such that
$\supp \varphi_v \subset \mi_E^{-m}$.
\end{lem}
\begin{proof}
Take a non-negative integer $m$ such that
$v$ is fixed by $(1+M_3(\mi_E^m))\cap G$.
For $a \in E^\times$ and $x \in \mi_E^{m}$,
we have
\begin{eqnarray*}
\varphi_v(a) & = & l(\pi(t(a))v) = l(\pi(t(a)u(x, -x\overline{x}/2))v)
= l(\pi(u(ax, -ax\overline{ax}/2)t(a))v)\\
& = &\psi_E(ax) l(\pi(t(a))v)
 = \psi_E(ax) \varphi_v(a)
\end{eqnarray*}
since $\pi(u(x, -x\overline{x}/2))v = v$.
This implies that $a\mi_E^{m} \subset \ri_E$
if $a \in \supp \varphi_v$.
So we get $\supp \varphi_v \subset \mi_E^{-m}$.
\end{proof}

\begin{cor}\label{cor:4.5}
Let $n$ be a non-negative integer.
For any $v \in V(n)$,
the function
$\varphi_v$ is $\ri_E^\times$-invariant 
and $\supp\varphi_v \subset \ri_E$.
\end{cor}
\begin{proof}
The function $\varphi_v$ is $\ri_E^\times$-invariant 
since $v$ is fixed by 
$t(a)$, $a \in \ri_E^\times$.
Applying the proof of Lemma~\ref{lem:varphi},
we get $\supp\varphi_v \subset \ri_E$.
\end{proof}

Let $\mathcal{C}_c^\infty(E^\times)$ be the subspace
of $\mathcal{C}^\infty(E^\times)$ consisting
of the compactly supported functions.

\begin{prop}\label{prop:killirov}
Let $v$ be an element in $V$ such that $p(v) \in V_{U_H}(Z_2)$.
Then $\varphi_v \in \mathcal{C}_c^\infty(E^\times)$.
\end{prop}
\begin{proof}
We set $V(U) = \langle \pi(u)v-v\, |\, v \in V, u \in U\rangle$.
Since $V_{U_H}(Z_2) = p(V(U))$,
it is enough to prove that
$\varphi_{\pi(u)v-v} \in  \mathcal{C}_c^\infty(E^\times)$,
for $v \in V$ and $u \in U$.

We write $u = u(x, y)$, where $x, y \in E$ such that $y+\overline{y}
+ x \overline{x} = 0$.
Let $a \in E^\times$ such that $\varphi_{\pi(u)v-v}(a) \neq 0$.
This implies 
$(\psi_E(ax)-1)\varphi_v(a) \neq 0$.
So we get $ax \not\in \ri_E$.
By Lemma~\ref{lem:varphi},
we conclude that $\supp\varphi_{\pi(u)v-v}$ is compact.
\end{proof}

By Proposition~\ref{prop:killirov},
we get a map
$V_{U_H}(Z_2) \rightarrow \mathcal{C}_c^\infty(E^\times);
p(v) \mapsto \varphi_v$.
\begin{lem}\label{lem:generic_P2}
The map
$V_{U_H}(Z_2) \rightarrow \mathcal{C}_c^\infty(E^\times);
p(v) \mapsto \varphi_v$ is an isomorphism
of $P_2$-modules.
\end{lem}
\begin{proof}
We claim that the map
$V_{U_H}(Z_2) \rightarrow \mathcal{C}_c^\infty(E^\times);
p(v) \mapsto \varphi_v$ is not zero.
Let $v$ be an element in $V$ such that $l(v) \neq 0$.
This implies $\varphi_v(1) \neq 0$.
Take an element $x$ in $E$ such that $\psi_E(x) \neq 1$
and
put $u = u(x, -x\overline{x}/2)$.
Then we have $\varphi_{\pi(u)v-v}(1)
= (\psi_E(x)-1)\varphi_v(1) \neq 0$.
So we conclude that 
the map
$V_{U_H}(Z_2) \rightarrow \mathcal{C}_c^\infty(E^\times);
p(v) \mapsto \varphi_v$ is not zero.

For $f \in \mathrm{ind}_{Z_2}^{P_2}(\psi)$,
we define $T(f) \in \mathcal{C}_c^\infty(E^\times)$
by
\[
T(f)(a) = 
f\left(
\begin{array}{cc}
a & 0\\
0 & 1
\end{array}
\right),\ a \in E^\times.
\]
Then we get an isomorphism of $P_2$-modules
$T: \mathrm{ind}_{Z_2}^{P_2}(\psi) \simeq \mathcal{C}_c^\infty(E^\times); f \mapsto T(f)$.
Since
$V_{U_H}(Z_2) \simeq 
\mathrm{ind}_{Z_2}^{P_2}(\psi) \simeq \mathcal{C}_c^\infty(E^\times)$
and $\mathrm{ind}_{Z_2}^{P_2}(\psi)$ is irreducible,
we see that
the map
$V_{U_H}(Z_2) \rightarrow \mathcal{C}_c^\infty(E^\times);
p(v) \mapsto \varphi_v$ is an isomorphism of $P_2$-modules.
\end{proof}

Now we get a criterion whether a vector in $V(n)$
lies in $\eta V(n-2)$,
in terms of Kirillov model.
\begin{lem}\label{lem:phi}
Let $n$ be an integer such that 
$n \geq 2$ and
$n> n_\pi$.
Suppose that $v \in V(n)$ satisfies $\supp\varphi_v \subset \mi_E$.
Then $v \in \eta V(n-2)$.
\end{lem}
\begin{proof}
It is easy to check that
the assumption implies
$\varphi_{(S-q^4)v} = 0$.
Since $U/U_H \simeq Z_2$,
we have 
$p((S-q^4)v) \in V_{U_H}(Z_2)$.
Thus
Lemma~\ref{lem:generic_P2}
says that $p((S-q^4)v) = 0$.
By Lemma~\ref{lem:pS},
we obtain
$v \in \eta V(n-2)$.
\end{proof}

Lemma~\ref{lem:phi} shows 
the following theorem, which bounds the growth of the dimensions
of oldforms.
\begin{thm}\label{thm:1}
Let $(\pi, V)$ be an irreducible generic representation of $G$.
For any non-negative integer $n$ such that $n+2 > n_\pi$,
we have
\[
\dim V(n+2) -\dim V(n) \leq 1.
\]
\end{thm}
\begin{proof}
Since the map $\eta: V(n) \rightarrow V(n+2)$ is injective,
it is enough to prove $\dim V(n+2) /\eta V(n) \leq 1$.
Let $v_1, v_2$ be elements in $V(n+2)\backslash \eta V(n)$.
Due to Corollary~\ref{cor:4.5},
the function $\varphi_{v_i}$  is $\ri_E^\times$-invariant
and 
$\supp\varphi_{v_i} \subset \ri_E$,
for $i = 1, 2$.
Put 
\[
\alpha = \varphi_{v_2}(1),\ \beta = \varphi_{v_1}(1).
\]
By Lemma~\ref{lem:phi},
we have $\alpha \neq 0$ and $\beta \neq 0$.
Since
$\supp\varphi_{\alpha v_1-\beta v_2} \subset \mi_E$,
Lemma~\ref{lem:phi}
implies $\alpha v_1-\beta v_2 \in \eta V(n)$.
So we conclude that 
$\dim V(n+2) /\eta V(n) \leq 1$.
\end{proof}

Applying Theorem~\ref{thm:1}
to $n = N_\pi-2$,
we obtain the uniqueness of newforms
for generic representations
whose conductors differ from those of their central characters.

\begin{thm}\label{thm:gen_1}
Suppose that an irreducible generic representation $(\pi, V)$
of $G$ satisfies
$N_\pi > n_\pi$ and $N_\pi \geq 2$.
Then we have
$\dim V(N_\pi) = 1$.
\end{thm}
\begin{proof}
By Theorem~\ref{thm:1},
we have
$\dim V(N_\pi) -\dim V(N_\pi-2) \leq 1$.
Since $\dim V(N_\pi) \geq 1$ and $\dim V(N_\pi-2) = 0$,
we get $\dim V(N_\pi) = 1$.
\end{proof}

We close this section by showing 
that the newform is a test vector for the Whittaker functional.
\begin{thm}\label{thm:test}
Let $\pi$ be an irreducible generic representation of $G$
such that $N_\pi \geq 2$ and $N_\pi > n_\pi$.
For a non-zero element  $v$ in $V(N_\pi)$,
we have $W_v(1) \neq 0$.
\end{thm}
\begin{proof}
The function $\varphi_v$ is supported by $\ri_E$.
Since $V(N_\pi-2) = \{0\}$,
it follows from Lemma~\ref{lem:phi}
that $\supp\varphi_v \not\subset \mi_E$.
Because the function $\varphi_v$ is $\ri_E^\times$-invariant,
we get $\varphi_v (1) = W_v(1) \neq 0$.
\end{proof}

\Section{Main theorems}\label{sec:main}
In this section,
we prove our main results.
We show that
the space of newforms is of dimension one.
Moreover,
we give a formula of the dimensions of the oldforms
for generic representations whose conductors are greater than
those of their central characters.

\subsection{Multiplicity one theorem of newforms}

As the well-known fact on $K_0$-fixed vectors,
we have the following:
\begin{prop}\label{prop:mult_1}
Let $(\pi, V)$ be an irreducible admissible representation of $G$.
Then we have $\dim V(1) \leq 1$.
\end{prop}
\begin{proof}
Let $(\pi, V)$ be an irreducible admissible representation of $G$
such that $V(1) \neq \{0\}$.
It is well known that $\pi$ can not be supercuspidal.
So $\pi$ can be embedded into a parabolically induced representation 
$\mathrm{Ind}_B^G\chi$,
for some quasi-character $\chi$ of $T$.
Let $(\mathrm{Ind}_B^G\chi)^{K_1}$ denote 
the space of the $K_1$-fixed vectors in $\mathrm{Ind}_B^G\chi$.
Then we have 
$\dim V(1) \leq \dim (\mathrm{Ind}_B^G\chi)^{K_1}$.
The dimension of $(\mathrm{Ind}_B^G\chi)^{K_1}$
equals to the number of the elements $g$ in $B\backslash G/K_1$
such that 
$\chi$ is trivial on $B\cap gK_1g^{-1}$.
So Lemma~\ref{lem:K_1} implies $\dim (\mathrm{Ind}_B^G\chi)^{K_1}
\leq 1$.
This completes the proof.
\end{proof}

We shall treat the case when 
Lemma~\ref{lem:pS} is not valid.
In this case,
we use the Hecke algebra isomorphism established by Moy.
\begin{thm}\label{thm:cent}
Let $(\pi, V)$ be an irreducible admissible  representation of $G$
which admits a newform.
Suppose that the conductor $N_\pi$ of $\pi$ satisfies $N_\pi \geq 2$ and
$N_\pi = n_\pi$. Then

(i) $\dim V(N_\pi) = 1$;

(ii) $\pi$ is not supercuspidal.
\end{thm}
\begin{proof}
(i)
We use the Hecke algebra isomorphism 
in \cite{U21}.
Put $n = N_\pi$.
We define two open compact subgroups of $G$ by
\[
P_{3n-3}
= \left(
\begin{array}{ccc}
1+\mi_E^{n-1} & \mi_E^{n-1} & \mi_E^{n-1}\\
\mi_E^n & 1+\mi_E^{n-1} & \mi_E^{n-1}\\
\mi_E^n & \mi_E^n & 1+\mi_E^{n-1} 
\end{array}
\right)\cap G,\
P_{3n-2}
= \left(
\begin{array}{ccc}
1+\mi_E^{n} & \mi_E^{n-1} & \mi_E^{n-1}\\
\mi_E^n & 1+\mi_E^{n} & \mi_E^{n-1}\\
\mi_E^n & \mi_E^n & 1+\mi_E^{n} 
\end{array}
\right)\cap G.
\]
Since $P_{3n-2} \subset K_n$,
all elements in $V(n)$ are fixed by $P_{3n-2}$.
We have $P_{3n-3} = Z_{n-1}(T_H\cap P_{3n-3})P_{3n-2}$
and $T_H\cap P_{3n-3} \subset K_n$.
So $P_{3n-3}$ acts on $V(n)$
by an extension $\rho$ of the central character
of $\pi$.

Let $\psi_F$ be a non-trivial additive character of $F$
with conductor $\mi_F$.
For an element $\beta \in M_3(E)$,
we define a map $\psi_\beta: M_3(E) \rightarrow \C^\times$ 
by
\[
\psi_\beta(x) = 
\psi_F(\mathrm{tr}_{E/F}\circ\mathrm{tr}(\beta(x-1))),\ x \in 
M_3(E).
\]
By \cite{Morris-2} Theorem 2.13,
there exists
\[
\beta
= \p^{1-n}
\left(
\begin{array}{ccc}
A & 0 & 0\\
0 & a\e & 0\\
0 & 0 & -\overline{A}
\end{array}
\right),\ A \in \ri_E/\mi_E,\ a \in \ri_F/\mi_F
\]
such that
\[
\rho(p) = \psi_\beta(p),\ p \in P_{3n-3}.
\]
Because $\rho$ is trivial on $T_H\cap P_{3n-3}$,
we have $A \in \mi_E$.
So we may assume $A = 0$.
Since we are assuming $n = n_\pi$,
we obtain $a \in \ri_F^\times$.

We define an open compact subgroup $J$ 
of $G$ as in 
\cite{U21} p. 191.
If $n = 2m$,
we put
\[
J 
= 
\left(
\begin{array}{ccc}
1+\mi_E^{n-1} & \mi_E^{m} & \mi_E^{n-1}\\
\mi_E^m & 1+\mi_E^{n-1} & \mi_E^{m}\\
\mi_E^n & \mi_E^m & 1+\mi_E^{n-1} 
\end{array}
\right)\cap G.
\]
If $n = 2m+1$,
we set
\[
J 
= 
\left(
\begin{array}{ccc}
1+\mi_E^{n-1} & \mi_E^{m} & \mi_E^{n-1}\\
\mi_E^{m+1} & 1+\mi_E^{n-1} & \mi_E^{m}\\
\mi_E^n & \mi_E^{m+1} & 1+\mi_E^{n-1} 
\end{array}
\right)\cap G.
\]
Then 
$J$ contains $P_{3n-3}$.
We can extend $\rho$ to
a character 
of $J$ which is trivial outside of $P_{3n-3}$,
which is also denoted by $\rho$.
We put $G' = HZ$,
$J' = J\cap G'$ and $\rho' = \rho|_{J'}$.
If we  denote by $\mathcal{H}(G//J, \rho)$,
$\mathcal{H}(G'//J', \rho')$ 
the Hecke algebras of $G$, $G'$ associated to 
$(J, \rho)$,
$(J', \rho')$ respectively, then
by \cite{U21} Corollary 2.8,
there exists a support-preserving algebra isomorphism
\begin{eqnarray}\label{eq:hecke0}
i: \mathcal{H}(G//J, \rho) \simeq 
\mathcal{H}(G'//J', \rho').
\end{eqnarray}

The $\rho$-isotypic component $\pi^\rho$
of 
$\pi$
 is an irreducible
$\mathcal{H}(G//J, \rho)$-module.
There is a left ideal $I$ of $\mathcal{H}(G//J, \rho)$
such that $\pi^\rho \simeq \mathcal{H}(G//J, \rho)/I$.
By (\ref{eq:hecke0}),
there is an irreducible admissible representation $\tau$ of $G'$
such that
$\tau^{\rho'} \simeq \mathcal{H}(G'//J', \rho')/i(I)$.
Since $G' = ZH$,
we can view $\tau$ as an irreducible admissible representation of 
$H \simeq \mathrm{U}(1,1)(E/F)$.
So there is an isomorphism $i: \pi^\rho \simeq \tau^{\rho'}$
such that
\[
i(\pi(f) v) = \tau(i(f)) i(v),\ v \in \pi^\rho,\
f \in \mathcal{H}(G//J, \rho).
\]

Replacing $J$ with $\eta^{m}J\eta^{-m}$,
we may assume $J \subset K_nZ_{n-1}$.
Then one can observe that $V(n) \subset \pi^\rho$.
Let $f_\rho$ be the identity element in 
$\mathcal{H}(G//J, \rho)$.
For $g \in G$,
we denote by $\delta_g$ the Dirac point mass at $g$.
Then we have
\begin{eqnarray}\label{eq:hecke}
\pi(f_\rho*\delta_k* f_\rho)v
=
\pi(f_\rho)\pi(k)\pi(f_\rho)v
= v,\ v \in V(n),\ k \in K_n\cap H.
\end{eqnarray}
Let $V'(n)$ be the image of $V(n)$
in $\tau^{\rho'}$.
We denote by $e_\rho$ the identity element in 
$\mathcal{H}(G'//J', \rho')$.
By (\ref{eq:hecke}), 
we get
\[
\tau(e_\rho*\delta_k* e_\rho)v'
= v',\ v' \in V'(n),\ k \in K_n\cap H.
\]
So we obtain
\begin{eqnarray*}
\tau(e_\rho)
\int_{K_n\cap H} \tau(k)v'dk
& = & 
\int_{K_n\cap H}\tau(e_\rho) \tau(k)\tau(e_\rho)v'dk\\
& = & 
\int_{K_n\cap H}\tau(e_\rho*\delta_k*e_\rho)v'dk\\
& = & 
\mathrm{vol}(K_n\cap H) v',
\end{eqnarray*}
for $v' \in V'(n)$.
This implies
\[
V'(n) \subset \tau(e_\rho)\tau^{K_n\cap H},
\]
where $\tau^{K_n\cap H}$
is the space of $K_n\cap H$-fixed vectors in $\tau$.
Since $K_n\cap H$ is a good maximal compact subgroup of 
$H$,
we obtain
$\dim V(n) = \dim V'(n) \leq \dim \tau^{K_n\cap H} \leq 1$.

(ii)
We have seen that 
$\tau$ has a non-zero $K_n\cap H$-fixed vector.
Since $K_n\cap H$ is a good maximal compact subgroup of 
$H$,
$\tau$ can not be supercuspidal.
As remarked in  \cite{U21} p. 195,
this implies
that $\pi$ is not supercuspidal.
\end{proof}

Let  $(\pi, V)$ be an
irreducible supercuspidal representation of $G$.
It is well known that 
$\dim V(0) = \dim V(1) = 0$.
Therefore we obtain the following
\begin{cor}
Let $\pi$ be an
irreducible supercuspidal representation of $G$.

(i) Suppose that $\pi$ is generic.
Then we have $N_\pi \geq 2$ and $N_\pi > n_\pi$.

(ii) If $\pi$ is not generic,
then $\pi$ has no $K_n$-fixed vectors for all $n \geq 0$.
\end{cor}
\begin{proof}
Part (i)  follows from Theorem~\ref{thm:cent} (ii).
Suppose that $\pi$ is non-generic and admits a newform.
Then, by Theorem~\ref{thm:cent} (ii),
we have $N_\pi \geq 2$ and $N_\pi > n_\pi$.
This contradicts Theorem~\ref{thm:non_gen}.
So if $\pi$ is not generic,
then $\pi$ has no $K_n$-fixed vectors for all $n \geq 0$.
This completes the proof of (ii).
\end{proof}

We shall prove our main theorem.
\begin{thm}\label{thm:mult_main}
Let $(\pi, V)$ be an irreducible admissible representation of $G$
which admits a newform.
Then the space $V(N_\pi)$ of newforms for $\pi$
is one-dimensional.
\end{thm}
\begin{proof}
Recall that $K_0$ is a good maximal compact subgroup of $G$.
It is well known that $\dim V(0) \leq 1$.
Due to Proposition~\ref{prop:mult_1},
we have $\dim V(1) \leq 1$.
So we may assume that $N_\pi \geq 2$.

If we further suppose that $N_\pi = n_\pi$,
then the assertion follows from 
Theorem~\ref{thm:cent} (i).
Suppose that $N_\pi > n_\pi$ and $N_\pi \geq 2$.
Then Theorem~\ref{thm:non_gen} says 
that  $\pi$ should be generic.
So we get $\dim V(N_\pi) = 1$
by Theorem~\ref{thm:gen_1}.
This completes the proof.
\end{proof}

\subsection{Oldforms for generic representations}

The following theorem bounds
the dimensions  of oldforms for generic representations.

\begin{prop}\label{prop:dim}
Let $(\pi, V)$ be an irreducible generic representation of $G$.
Then we have
\[
\dim V(n) \leq \left\lfloor \frac{n-N_\pi}{2}\right\rfloor +1,
\]
for $n \geq N_\pi$.
\end{prop}
\begin{proof}
By Theorem~\ref{thm:mult_main},
we have $\dim V(N_\pi) = 1$.
We claim that $\dim V(N_\pi+1) \leq 1$.
If $N_\pi = 0$,
we have $\dim V(N_\pi+1) \leq 1$
by Proposition~\ref{prop:mult_1}.
Suppose that $N_\pi \geq 1$.
Then, by Theorem~\ref{thm:1},
we obtain
$\dim V(N_\pi+1) -\dim V(N_\pi-1) \leq 1$.
So we get $\dim V(N_\pi+1) \leq 1$.

For $\delta \in \{0, 1\}$ and $k \in \Z_{\geq 0}$,
we have
$\dim V(N_\pi+\delta +2k) \leq k +\dim V(N_\pi+ \delta)
\leq k + 1$,
by Theorem~\ref{thm:1}.
This completes the proof.
\end{proof}
We give a basis for oldforms
for generic representations $\pi$
which
satisfy $N_\pi \geq 2$ and $N_\pi > n_\pi$.

\begin{thm}\label{thm:main3}
Let $(\pi, V)$ be an irreducible generic representation of $G$.
Suppose that $W_v(1) \neq 0$
for all non-zero elements $v$ in $V(N_\pi)$.
Then, for $n \geq N_\pi$,
the set
$\{\theta^{'i} \eta^j v\ |\ i+2j+N_\pi = n\}$
forms a basis for $V(n)$.
In particular,
\[
\dim V(n) = \left\lfloor \frac{n-N_\pi}{2}\right\rfloor +1.
\]
\end{thm}
\begin{proof}
By Proposition~\ref{prop:dim},
it is enough to prove that
$\{\theta^{'i} \eta^j v\ |\ i+2j+N_\pi = n\}$
is linearly independent.
Let $m$ be a non-negative integer.
For $v \in V(m)$,
we have
\[
\varphi_{\eta v}(1) = \varphi_v(\p^{-1}),\
\varphi_{\theta' v}(1) = \varphi_v(\p^{-1})
+ q\varphi_v(1),
\]
by Proposition~\ref{prop:theta}.
Due to Corollary~\ref{cor:4.5},
we obtain 
\[
\varphi_{\eta v}(1) = 0,\
\varphi_{\theta' v}(1) = 
 q\varphi_v(1).
\]
So for $i, j \geq 0$,
we get
\[
\varphi_{\theta^{'i} \eta^j v}(1) 
=
\left\{
\begin{array}{cc}
q^{i}\varphi_v(1), & \mathrm{if}\ j = 0;\\
0, & \mathrm{otherwise}.
\end{array}
\right.
\]

Let $v$ be a non-zero element in $V(N_\pi)$.
Suppose that 
$\sum_{i+2j+N_\pi = n}\alpha_j\theta^{'i} \eta^j v = 0$ ($\alpha_j \in \C$).
Then we have
\[
0 = \varphi_{\sum_{j}\alpha_j\theta^{'i} \eta^j v }(1)
= \alpha_0 q^{n-N_{\pi}}\varphi_v(1),
\]
so that $\alpha_0 = 0$ by assumption.
So we get 
$\sum_{i+2j+N_\pi = n, j \geq  1}\alpha_j\theta^{'i} \eta^j v = 0$.
Since $\eta$ is injective and commutes with $\theta'$,
we have
$\sum_{i+2j+N_\pi = n, j \geq  1}\alpha_j\theta^{'i} \eta^{j-1} v = 0$.
Repeating this argument,
we obtain $\alpha_j = 0$, for all $j$.
\end{proof}

\begin{rem}
Suppose that an irreducible generic representation $\pi$
of $G$ satisfies 
$N_\pi \geq 2$ and $N_\pi > n_\pi$.
Then Theorem~\ref{thm:test}
says that
the assumption of Theorem~\ref{thm:main3} holds for $\pi$.
\end{rem}

\subsection{Oldforms for non-generic representations}
We close this paper with a result on the 
possibilities of the dimensions of oldforms for 
non-generic representations.

\begin{thm}
Let $(\pi, V)$ be an irreducible non-generic representation of 
$G$
which admits a newform.
For any non-negative integer $k$,
the following holds.

(i)
$\dim V(N_\pi+2k) = 1$.

(ii)
If $N_\pi \geq 1$,
then
$\dim V(N_\pi+2k+1) = 0$.

(iii)
Suppose that $N_\pi = 0$.
Then
$\dim V(2k+1) = \dim V(1)$.
\end{thm}
\begin{proof}
(i)
Lemma~\ref{lem:non-generic} implies
$V(N_\pi+2k) = \eta^k V(N_\pi)$.
So the assertion follows from Theorem~\ref{thm:mult_main}.
(ii)
Suppose that $N_\pi \geq 1$.
Lemma~\ref{lem:non-generic} says that
$V(N_\pi+2k+1) = \eta^{k+1} V(N_\pi-1) = \{0\}$.
(iii)
By
Lemma~\ref{lem:non-generic}, we have
$V(2k+1) = \eta^{k} V(1)$.
\end{proof}


\begin{thebibliography}{1}

\bibitem{Baruch}
E.~M. Baruch.
\newblock On the gamma factors attached to representations of {${\rm U}(2,1)$}
  over a {$p$}-adic field.
\newblock {\em Israel J. Math.}, 102:317--345, 1997.

\bibitem{BZ}
I.~N. Bern{\v{s}}te{\u\i}n and A.~V. Zelevinski{\u\i}.
\newblock Representations of the group {$GL(n,F),$} where {$F$} is a local
  non-{A}rchimedean field.
\newblock {\em Uspehi Mat. Nauk}, 31(3(189)):5--70, 1976.

\bibitem{Casselman}
W.~Casselman.
\newblock On some results of {A}tkin and {L}ehner.
\newblock {\em Math. Ann.}, 201:301--314, 1973.

\bibitem{JPSS}
H.~Jacquet, I.~Piatetski-Shapiro, and J.~Shalika.
\newblock Conducteur des repr\'esentations du groupe lin\'eaire.
\newblock {\em Math. Ann.}, 256(2):199--214, 1981.

\bibitem{LR}
J.~Lansky and A.~Raghuram.
\newblock Conductors and newforms for {$U(1,1)$}.
\newblock {\em Proc. Indian Acad. Sci. (Math Sci.)}, 114(4):319--343, 2004.

\bibitem{Morris-2}
L.~Morris.
\newblock Tamely ramified supercuspidal representations of classical groups.
  {I}. {F}iltrations.
\newblock {\em Ann. Sci. \'Ecole Norm. Sup. (4)}, 24:705--738, 1991.

\bibitem{U21}
A.~Moy.
\newblock Representations of ${U}(2,1)$ over a $p$-adic field.
\newblock {\em J. Reine Angew. Math.}, 372:178--208, 1986.

\bibitem{Reeder}
M.~Reeder.
\newblock Old forms on {${\rm GL}_n$}.
\newblock {\em Amer. J. Math.}, 113(5):911--930, 1991.

\bibitem{RS}
B.~Roberts and R.~Schmidt.
\newblock {\em Local newforms for {GS}p(4)}, volume 1918 of {\em Lecture Notes
  in Mathematics}.
\newblock Springer, Berlin, 2007.

\end{thebibliography}
\end{document}